\documentclass[12pt]{amsart}
\title[Stablility of conjugation-invariant word norms]{The stable conjugation-invariant word norm is rational in free groups}
\author{Henry Jaspars}

\usepackage[T1]{fontenc}
\usepackage{amsmath,amsthm,amssymb}
\usepackage{mathtools}
\usepackage{enumitem}
\usepackage{hyperref}
\usepackage{cleveref}

\newenvironment{claim}[1][]{\par\noindent\textit{Claim.}\space}{}
\newenvironment{claimproof}[1][]{\par\noindent\textit{Proof of claim.}\space\pushQED{\qed}}{\popQED}
\usepackage[all]{xy}

\newtheorem{theorem}[subsection]{Theorem}
\newtheorem{lemma}[subsection]{Lemma}
\newtheorem{prop}[subsection]{Proposition}

\newtheorem{cor}[subsection]{Corollary}

\theoremstyle{definition}
\newtheorem{defn}[subsection]{Definition}
\newtheorem{example}[subsection]{Example}
\newtheorem{examples}[subsection]{Examples}

\theoremstyle{remark}
\newtheorem{remark}[subsection]{Remark}

\newtheorem*{ack*}{Acknowledgements}

\numberwithin{equation}{section}
\numberwithin{figure}{section}

\newcommand{\B}[1]{{\mathbf #1}}
\newcommand{\C}[1]{{\mathcal #1}}

\newcommand{\OP}{\operatorname}
\newcommand{\col}{\vcentcolon}

\begin{document}
\begin{abstract}
We establish the rationality of the stable conjugation-invariant word norm on free groups and virtually free Coxeter groups.
\end{abstract}
\maketitle
\section{Introduction}\label{S:intro}

A \textit{norm} on a group $G$ is a non-negative 
function $\nu\colon G\to \B R$ such that for all $g, h \in G$,
\begin{enumerate}
    \item $\nu(g) = 0$ if and only if $g=1$;
    \item $\nu(gh) \leq \nu(g) + \nu(h)$.
\end{enumerate}
Moreover, if $\nu(ghg^{-1})=\nu(h)$ for all $g,h\in G$, then
the norm is called \textit{conjugation-invariant}. Examples of conjugation-invariant norms include the Hofer norm
on the group of Hamiltonian diffeomorphisms of a symplectic
manifold \cite{sympletic}, the commutator length \cite{scl}, and other verbal norms \cite{wlength}, as well as word norms associated with generating sets invariant
under conjugations \cite{jarek}, which in the context of Coxeter groups are known as \textit{reflection length}, and have been studied extensively \cite{cox1, cox2, cox3}. In a combinatorial setting, conjugation-invariant norms also bear relation to non-crossing matchings and enumerating unpaired nucleotides in secondary RNA structures \cite{rna}.

The \textit{stable (or translation) length} of $g\in G$ with respect to
$\nu$ is the limit $$\tau_{\nu}(g)\col=\lim_{k\to \infty} \frac{\nu(g^k)}{k}.$$
In this paper we prove that the stable
length with respect to the conjugation-invariant word norm on certain virtually free groups is rational. 

More precisely, throughout, let $G$ be a group normally generated by
$S$. Let $\|g\|_S$ denote the
\textit{conjugation-invariant word norm} associated with $\bar{S}= \bigcup_{s \in S} {\rm
Conj}(s)$, where ${\rm Conj}(s)$ denotes the conjugacy class of $s$, namely, 
$$\|g\|_S \vcentcolon= \inf\{\ell: g_1 s_1 g_1^{-1} \hdots g_\ell s_\ell g_\ell ^{-1} = g,\, g_i \in G, s_i \in S\}.$$
Also define the \textit{cancellation length} of a word $w$ with characters in $S$ to be the minimum number of letters which need to be deleted to obtain a word representing the trivial element. It is known that if $G = \langle X\, |\, R \rangle$ is a free group freely generated by $X$ or a virtually free Coxeter group with standard Coxeter presentation, then for a word $w$ with characters in $S$ representing an element $g \in G$, the cancellation norm of $w$ is equal to the conjugation-invariant norm $\|g\|_S$, and in particular is invariant upon the choice of $w$ representing $g$ \cite[Proposition 2.E']{jarek}.

The question of rationality for stable word norms over finitely presented groups is a subject of interest in geometric group theory and analysis. In \cite{rationality}, Calegari showed that the stable commutator norm over free groups is rational using techniques from topology and geometry. It is also known that rationality of the stable commutator norm does not hold for general finitely presented groups \cite{zhuang}, proving in the negative a conjecture of Gromov \cite{gromov}. This motivates the analogous study of the stable conjugation-invariant word norm, for which we prove the following result.

\begin{theorem}\label{T:main}
Let $G = \langle X \, | \, R\rangle$ be a free group freely generated by $X$ or a virtually free Coxeter group with the standard Coxeter presentation respectively. Let $S = X^{\pm 1}$, and further let $g_1,\ldots,g_n\in G$.
Then the sequence 
$(\|g_1^k\hdots g_n^k\|_S)_k$
is uniformly semi-arithmetic
in the sense of \Cref{D:semi-arithmetic}. Moreover, 
$$\tau_S(g_1, \hdots, g_n) \col= \lim_{k\to \infty} \frac{\|g_1^k\hdots g_n^k\|_S}{n}$$ 
exists and is equal to $\frac{d}{m} \in \B Q$, where $(\|g_1^k\hdots g_n^k\|_S)_k$ has period $m>0$ and difference $d\geq 0$.
\end{theorem}

\begin{example}
Consider the free product of $n$ cyclic groups of order 2, $\B Z/2\B Z* \dots *\B Z/2\B Z$. Since this is a virtually free Coxeter group admitting a standard Coxeter presentation with generating set $X = \{s_1, \hdots, s_n\}$ (where $s_i$ is the generator of the $i$th copy of $\B Z/2 \B Z$), \Cref{T:main} applies in this instance. Virtually free Coxeter groups are characterised in \cite[Proposition 8.8.5]{coxeter}. 
\end{example}

We will illustrate the result with three examples in the case where $F = \langle a, b \rangle$ is the free group of rank 2 with free generating set $S = \{a^{\pm 1}, b^{\pm 1}\}$.
\\ \\
\begin{examples}
\mbox{}
\begin{enumerate}
\item For $a^n b^m \in F$, then $\|g^k\| = k\cdot(n + m)$, and coincides with the standard word norm. Hence $\tau_S(g) = n + m$. 

\item Consider $[a, b] \in F$. Let $\psi: F \rightarrow \B Z^2$ be the abelianisation with $\psi(a) = (1,0)$ and $\psi(b) = (0,1)$, and for reduced $g = s_1 \hdots s_n$ with $s_i \in S$, let $g_k = s_1 \hdots s_k$. Define quasimorphism $\xi: F \rightarrow \B Z$ where $\xi(g)$ is be the number of left turns minus the number of right turns in the path $(\psi(g_k))_{k = 1}^n$ (as introduced in \cite[Example 6.5]{quasi}). It is immediate that $\xi([a, b]^k) = 4k - 1$, that $|\xi(gh) - \xi(g) - \xi(h)| \leq 2$, and $|\xi(gsg^{-1})|\leq 2$ for all $g, h \in F$ and $s \in S$. Hence, if $h = g_1 s_1 g_1^{-1} \hdots g_\ell s_\ell g_\ell ^{-1}$, by induction, $|\xi(h)| \leq 4 \ell - 2$, and hence $4k - 1 = |\xi([a, b]^k)| \leq 4 \|[a, b]^k\|_S - 2$. Therefore, we have that $\|[a, b]^k\|_S \geq 2 \lfloor k/2 \rfloor + 2$.

To see that this is sharp, observe that by deleting the first $\lfloor k / 2 \rfloor$ $b$'s and the final $\lfloor k / 2 \rfloor$ $b^{-1}$'s in the reduced word of $[a, b]^k$, we obtain $b^{-\lfloor k/2 \rfloor} a b^{\lfloor k/2 \rfloor} a^{-1}$ for $k$ even and $b^{-\lfloor k/2 \rfloor} a b^{\lceil k/2 \rceil} a^{-1} b^{-1}$ otherwise, which both have cancellation length 2. Therefore $\|[a, b]^k\|_S = 2 \lfloor k / 2 \rfloor + 2$ and moreover $\tau_S([a, b]) = 1$. In particular, in this case, the sequence is not arithmetic.

\item For $a^m b a^{-m}, b^{-1} \in F$, since $(a^m ba^{-m})^k (b^{-1})^k = a^m b^k a^{-m} b^{-k}$, it follows that $\|(a^m ba^{-m})^k (b^{-1})^k\| = 2 \min\{m, k\}$ and hence $\tau_S(a^m b a^{-m}, b^{-1}) = 0$.
\end{enumerate}
\end{examples}

\begin{remark}
These examples suggest that \ref{T:main} cannot be strengthened, as we cannot replace uniform semi-arithmeticity with mere arithmeticity.
\end{remark}

\begin{ack*}
The author would like to offer thanks to both Jarek K\k{e}dra and Assaf Libman, without whose indispensable mentorship, advice, time and support this paper would not have been possible. The author would also like to thank the referee for drawing the author's attention to the literature on reflection lengths.
\end{ack*}

\section{Semi-arithmetic sequences}

Let $\widehat{\B N} \col= \B N \cup \{\infty\}$ with the usual total order, where $x+\infty=\infty+x=\infty$ for any $x \in \widehat{\B N}$. Moreover, let an \textit{arithmetic sequence with difference $d \geq 0$} have the usual meaning. Observe that the constant infinite sequence $(a_k)_k$ where $a_k = \infty$ for $k \geq 0$ is arithmetic with difference $d$ for any $d \geq 0$.

\begin{defn}\label{D:semi-arithmetic}
Let $(a_k)_k$ be a sequence in $\widehat{\B N}$. Then, it is:
\begin{enumerate}
\item
\textit{Eventually finite} if $a_k \in \B N$ for all $k$ sufficiently large;

\item
\textit{Eventually arithmetic} if $(a_{k + n})_k$ is arithmetic for some $n \geq 0$;

\item
\textit{Semi-arithmetic} with period $m > 0$ if for all $n \geq 0$, the sequences $(a_{mk + n})_k$ are eventually arithmetic with differences $d_n$; 

\item
{\em Uniformly semi-arithmetic} with period $m>0$ and difference $d \geq 0$ if for all $n \geq 0$, the sequences $(a_{mk + n})_k$ are eventually arithmetic with difference $d$.
\end{enumerate}
\end{defn}

\begin{remark}\label{helpfulrem}
Let $(a_k)_k$ be a semi-arithmetic sequence with period $m$. Since, for all $n \geq 0$, $(a_{mk + n + m})_k$ is a tail of $(a_{mk + n})_k$, it follows that $d_n = d_{n + m}$ for all $n \geq 0$.
\end{remark}

The following two results are immediate consequences of \Cref{D:semi-arithmetic}:

\begin{prop}\label{L:min semi-arithmetic}
Let $(a_k)_k$ and $(a'_k)_k$ be semi-arithmetic sequences with periods $m$ and $m'$ respectively. Then the sequence $(\min\{ a_k,a'_k\})_k$ is semi-arithmetic with period $mm'$. \qed
\end{prop}

\begin{prop}\label{L:usa limit}
Let $(a_k)_k$ be a eventually finite uniformly semi-arithmetic sequence with period $m > 0$ and difference $d \geq 0$.
Then
$$\lim_{k \to \infty} \frac{a_k}{k}$$
exists and is equal to $\frac{d}{m} \in \B Q$. \qed
\end{prop}

\begin{lemma}\label{L:usa criterion}
Let $(a_k)_k$ be semi-arithmetic.
Suppose there exists $D \geq 0$ such that $a_{k+1} \leq a_k +D$ for all $k \geq 0$.
Then $(a_k)_k$ is uniformly semi-arithmetic.
\end{lemma}

\begin{proof}
If $a_k=\infty$ for all $k\geq 0$, the result is immediate. Otherwise, $(a_k)_k$ is eventually finite. Hence, without loss of generality, assume $a_k<\infty$ for all $k$.

Let $(a_k)_k$ have period $m > 0$. By \Cref{D:semi-arithmetic}, for all $n \geq 0$, the sequences $(a_{mk+n})_k$ are arithmetic with some difference $d_n \geq 0$. By the hypothesis, the sequence $(a_{mk+n+1}-a_{mk+n})_k$ is bounded above by $D$. However, observe that  $(a_{mk+n+1}-a_{mk+n})_k$ is eventually arithmetic with difference $d_{n+1}-d_n$. Therefore, $d_{n+1}-d_n \leq 0$, and hence by \Cref{helpfulrem}, $$d_0 \geq d_1 \geq \hdots \geq d_m = d_0.$$
Hence by \Cref{helpfulrem}, $(d_n)_n$ is constant, and $(a_k)_k$ is uniformly semi-arithmetic.
\end{proof}

\section{Semilinear sets}

Let $\B N^X$ denote the free abelian monoid over $X$, or equivalently, the set of all functions $X \to \B N$.

\begin{defn}[Compare {\cite[Definition~3]{tame}}]\label{D:semilinear}
A set $\Omega \subseteq \B N^X$ is called \textit{linear} if it is of the form $v + M$ where $v \in \B N^X$ and $M \subseteq \B N^X$ is a finitely-generated submonoid of $\B N^X$.
It called is \textit{semilinear} if it is a union of finitely many linear subsets of $\B N^X$.
\end{defn}

The following propositions are immediate consequences of \Cref{D:semilinear}:

\begin{prop}\label{L:homomorphic image semilinear}
Let $\phi \colon \B N^X \to \B N^Y$ be an homomorphism. Then if $\Omega \subseteq \B N^X$ is semilinear, $\phi(\Omega) \subseteq \B N^Y$ is also semilinear. \qed
\end{prop}

\begin{prop}\label{products}
If $\Omega_X \subseteq \B N^X$ and $\Omega_Y \subseteq \B N^Y$ are semilinear, then $\Omega_X \times \Omega_Y \subseteq \B N^X \times \B N^Y$ is also semilinear. \qed
\end{prop}

We also state the following seminal result in the study of semilinear sets.

\begin{lemma}[Liu-Weiner, {\cite[Theorem 1]{tame}}]\label{T:intersection semilinear}
The family of semilinear subsets of $\B N^X$ is closed under intersections. \qed
\end{lemma}

\begin{example}\label{E:diagonal}
Let $X$ be a finite set, and let $\mathbf{1}_X \colon X \to \B N$ denote the characteristic function of $X$. Then, define the \textit{diagonal over $X$} to be the set $\Delta_X \col= \{k \cdot \mathbf{1}_X : k \in \B N\} \subseteq \B N^X$. In particular, observe that $\Delta_X$ is semilinear. 
\end{example}

\begin{defn}\label{D:Env}
Let $\Omega \subseteq \B N^2$.
For $k \in \B N$, let $\Omega_{(k)} \col= \{ \ell : (k,\ell) \in \Omega\}$.
Then the \textit{lower envelope} of the set $\Omega$ is the sequence
$$\OP{Env}(\Omega) \col= (\inf \Omega_{(k)})_k,$$
where, by convention, the infimum of the empty set is $\infty$.
\end{defn}

The following proposition follows immediately from \Cref{D:Env}.

\begin{prop}\label{L:Env of union}
Let $\Omega, \Omega' \subseteq \B N^2$, and let $(a_k)_k=\OP{Env}(\Omega)$ and $(a'_k)_k=\OP{Env}(\Omega')$ respectively.
Then $\OP{Env}(\Omega \cup \Omega') = (\min\{a_k,a'_k\})_k$. \qed
\end{prop}

\begin{prop}\label{L:discard zeros for Env}
Let $\Omega$ be the submonoid of $\B N^2$ generated by a finite set $V=\{(x_\alpha,y_\alpha)\}_{\alpha \in A} \subseteq \B N^2$.
Set $A'\col=\{ \alpha \in A : x_\alpha \neq 0\}$ and $\Omega'$ be the submonoid generated by $\{(x_\alpha,y_\alpha) \}_{\alpha \in A'}$.
Then $\OP{Env}(\Omega)=\OP{Env}(\Omega')$.
\end{prop}

\begin{proof}
Since $\Omega' \subseteq \Omega$, it follows that $\inf \Omega_{(k)} \leq \inf \Omega'_{(k)}$.
If $\inf \Omega_{(k)}=\infty$, then clearly $\inf \Omega_{(k)} = \inf \Omega_{(k)} '$.
Otherwise, without loss of generality, let $k \geq 1$ be such that $\inf \Omega_{(k)}< \infty$.
Then, there exists $\lambda_\alpha \in \B N$ for all $\alpha \in A$ such that $\sum_{\alpha \in A} \lambda_\alpha x_\alpha = k$ and $\sum_{\alpha \in A} \lambda_\alpha y_\alpha = \inf \Omega_{(k)}$.
Since $\sum_{\alpha \in A'} \lambda_\alpha y_\alpha$ is an element of $\Omega'_{(k)}$, and $\sum_{\alpha \in A'} \lambda_\alpha y_\alpha \leq \sum_{\alpha \in A} \lambda_\alpha y_\alpha$, by minimality, it follows that $\inf \Omega'_{(k)} \leq \inf \Omega_{(k)}$ for all $k$, from which equality follows.
\end{proof}

We now state the following classical lemma.

\begin{lemma}[Frobenius, {\cite[Theorem~3.15.2]{genfunc}}] \label{L:coins prop}
Let $x_1,\hdots,x_n \in \B N$, and let $g = \gcd(x_1, \hdots, x_n)$.
Then there exists $N \geq 0$ such that for all $k \geq N$ with $g \ | \ k$, there exists $\lambda_1,\hdots,\lambda_n \in \B N$ such that
\[
\pushQED{\qed} 
k=\sum_i \lambda_i x_i. \qedhere
\popQED
\]   
\end{lemma}

We arrive at the main result of this section, whose proof is adapted and simplified from \cite[Theorem, p. 329]{greenberg}.

\begin{lemma}\label{P:Env of submonoid}
Let $\Omega \subseteq \B N^2$ be the monoid generated by a finite set $V \subseteq \B N^2$.
Then $\OP{Env}(\Omega)$ is uniformly semi-arithmetic.
\end{lemma}

\begin{proof}
Let $V \col= \{(x_i,y_i) : 1 \leq i \leq n\}$, and let $(a_k)_k \col= \OP{Env}(\Omega)$.
By \Cref{L:discard zeros for Env}, we may assume that $x_i \neq 0$ for all $i$, and as the result is trivial if $n=0$, we may assume $n \geq 1$. By reordering the generating set, we may also assume that 
$$
\frac{y_1}{x_1} \leq \frac{y_i}{x_i} 
$$
for all $1 \leq i \leq n$.

First, consider the case when $\gcd(x_1,\hdots,x_n)=1$.
By \Cref{L:coins prop}, there exists $N \geq 0$ such that for all $k \geq N$, $a_k < \infty$. Let $k \geq 0$ be such that $a_k < \infty$, and choose $\lambda_i \in \B N$ for all $1 \leq i \leq n$ such that $\sum_i \lambda_i x_i = k$ and $\sum_i \lambda_i y_i = a_k$. Then
\begin{equation}\label{E:Env of submonoid:1}
a_k = \sum_i \lambda_i y_i  = \sum_i \lambda_i x_i \cdot \frac{y_i}{x_i} \geq \sum_i \lambda_i x_i \cdot \frac{y_1}{x_1} = k \cdot\frac{y_1}{x_1}.
\end{equation}
Define $\lambda'_i$ by 
$$\lambda'_i \col= \begin{cases}
    \lambda_i +1 & i = 1\\
    \lambda_i & i > 1.
\end{cases}$$
Then $\sum_i \lambda'_i x_i = k+x_1$ and $\sum_i \lambda'_i y_i = a_k+y_1$.
By the minimality of $a_k$,
\begin{equation}\label{E:Env of submonoid:2}
a_{k+x_1} \leq a_k+y_1.
\end{equation}
Let $N \leq \ell < N+x_1$. Consider the sequence
$$
(b_j)_{j \geq 0} \col= \left(a_{\ell+x_1j} - (\ell+x_1j) \cdot \frac{y_1}{x_1}\right)_{j \geq 0}.
$$
Observe that $(b_j)_j$ is decreasing, since by \eqref{E:Env of submonoid:2}
$$
b_{j+1}-b_j =a_{\ell+x_1j+x_1}-a_{\ell+x_1j}-y_1 \leq 0
$$
for all $j \geq 0$. Moreover, by \eqref{E:Env of submonoid:1}, $b_j \geq 0$. Therefore $(b_j)_j$ is eventually constant.
Hence there exists a constant $c$ such that for all $j$ sufficiently large,
$$
a_{\ell+x_1j}=c+y_1j.
$$
Therefore $(a_{\ell + x_ij})_j$ is eventually arithmetic with difference $y_1$. Since $N \leq \ell < N + x_1$ was arbitrary, it follows that $(a_k)_k$ is uniformly semi-arithmetic with period $x_1$ and difference $y_1$.

In the general case, where $\gcd(x_1,\hdots,x_n) = g > 1$,
set $x'_i\col=x_i/g$, let $\Omega'$ be the submonoid generated by $V' \col= \{(x_i', y_i): 1 \leq i \leq n\}$.
Since $\gcd(x_1',\hdots,x_n')=1$ and let $(a_k')_k\col=\OP{Env}(\Omega')$. Then $(a_k)_k$ is uniformly semi-arithmetic with period $x_1'$ and difference $y_1$.
Clearly,
$$
a'_k=\left\{
\begin{array}{ll}
a_{k/g} & \text{if } g \ | \ k \\
\infty  & \text{otherwise.}
\end{array}
\right.
$$
Hence, if $g \ | \ k$, then $(a_{k+x_1j})_j = (a'_{k/g+x_1'j})_j$ is eventually arithmetic with difference $y_1$. Otherwise, $a_{k+x_1j} = \infty$ for all $j \geq 0$, and hence $(a_{k+x_1j})_j$ is an arithmetic sequence in $\widehat{\B N}$. 
It follows that $(a_k)_k$ is uniformly semi-arithmetic with period $x_1$ and difference $y_1$.
\end{proof}

\begin{cor}\label{C:Env of semilinear}
Let $\Omega \subseteq \B N^2$ be semilinear.
Then $\OP{Env}(\Omega)$ is semi-arithmetic.
\end{cor}

\begin{proof}
Let $M$ be a finitely generated submonoid of $\B N^2$. By \Cref{P:Env of submonoid}, $(a_k)_k=\OP{Env}(M)$ is a uniformly semi-arithmetic sequence in $\B N$.
It can be seen that if $v=(x_0,y_0) \in \B N^2$, then for the linear set $v+M$,
$$
\OP{Env}(v+M)_{k} = \begin{cases}
    \infty & \text{if } k < x_0 \\
    a_{k-x_0}+y_0 & \text{otherwise.} 
\end{cases}
$$
Therefore $\OP{Env}(v+M)$ is uniformly semi-arithmetic. Since a semilinear subset of $\B N^2$ is the union of finitely many linear sets, the result follows by repeated application of \Cref{L:Env of union} and \Cref{L:min semi-arithmetic}.
\end{proof}

Observe that, even if $(a_k)_k$ and $(a'_k)_k$ are uniformly semi-arithmetic, is is possible that the sequence $(\min\{a_k,a'_k\})_k$ is non-uniformly semi-arithmetic. Thus, \Cref{C:Env of semilinear} cannot be improved to uniform semi-arithmeticity.

\section{Formal Languages}

Let $X^*$ denote the free monoid generated by a set $X$, or equivalently, the set of all finite sequences of elements in $X$ with concatenation as monoidal structure.
Call the elements of $X^*$ {\em words} in the {\em alphabet} $X$, and denote the empty word by $\varepsilon$.
For $w \in X^*$, define \textit{length function} $|\cdot| \colon X^* \rightarrow \B N$ to be the unique homomorphism defined on generators by $|x| = 1$ for all $x \in X$. Observe that
\begin{equation} \label{lengths}
    |w_1^{k_1} \hdots w_n^{k_n}| = \sum_{i=1}^n k_i \cdot |w_i|. 
\end{equation}
A subset $\C L \subseteq X^*$ is called a {\em language} on the alphabet $X$.
Given languages $\C L, \C M \subseteq X^*$ let $\C L \cdot \C M$ denote the language
$$\C L \cdot \C M \col=\{ uv \, : \, u \in \C L, v \in \C M\}.$$
Write $\C L^n$ for the $n$-fold concatenation $\C L \cdot \hdots \cdot \C L$, where by convention $\C L^0\col=\{\varepsilon\}$. Further define the \textit{Kleene star} of $\C L$ to be $$\C L^* \col= \bigcup_{n \geq 0} \C L^n.$$

\begin{defn}\label{proj}
    Let $X,Y$ be disjoint finite alphabets.
Then the \textit{projection homomorphism} of free monoids is the homomorphism
\begin{equation*}
\pi_X \colon (X \sqcup Y)^* \to X^*
\end{equation*}
defined on generators by $\pi_X(x)=x$ for all $x \in X$ and $\pi_X(y)=\varepsilon$ for all $y \in Y$.
\end{defn}

\begin{defn}\label{D:P(w)}
Given a word $w \in X^*$, a {\em subword} is a subsequence of $w$.
Let $\mathcal{P}(w)$ denote the finite language consisting of the subwords of $w$.
Observe that, for words $w_1,\hdots,w_n \in X^*$,
\begin{equation}\label{E:subwords concatenation}
\mathcal{P}(w_1\hdots w_n) = \mathcal{P}(w_1) \cdot \hdots \cdot \mathcal{P}(w_n).
\end{equation}
\end{defn}

\begin{defn}\label{D:L pseudo-norm}
Let $\C L \subseteq X^*$ be a language. 
The {\em cancellation length relative to $\C L$} of $w \in X^*$ is defined as
$$\| w \|_{\C L} \col= \inf \{ |w|-|u| : u \in \mathcal{P}(w) \cap \C L\}.$$
Equivalently, $\|w\|_{\C L}$ is the infimum of all $\ell$ such that it is possible to delete $\ell$ characters from $w$ to yield an element of $\C L$.
\end{defn}

\begin{defn}[Compare {\cite[Section~3.1.1]{hopcroft}, \cite[Theorem~3.1]{shamir}}]\label{D:Regular language}
The collection of \textit{regular} languages over the alphabet $X$ is the smallest collection of languages satisfying the following:
\begin{enumerate}
\item If $\C L \subseteq X^*$ is finite, then $\C L$ is regular;
\item If $\C L$ and $\C M$ are regular, then $\C L^*, \C L \cup \C M $ and $\C L \cdot \C M$ are also regular.
\end{enumerate}
\end{defn}

\begin{defn}[Compare {\cite[Section~5.12]{hopcroft}, \cite[Section~19]{chomsky}}]\label{D:CFL}
A {\em context-free grammar} (CFG) is a quadruple $G=(V,X,R,v_0)$, where $X \subseteq V$ are finite sets of characters, $v_0 \in V \setminus X$ and $R \subseteq (V \setminus X) \times V^*$ is a finite set. Let $\C L \subseteq V^*$ be the smallest language containing the word $v_0$ such that, if $w_0 vw_1 \in \C L$ with $w_0,w_1 \in V^*$ and $v \in V \setminus X$, and if $(v,v') \in R$, then $w_0 v' w_1 \in \C L$. The language $\C L' = \C L \cap X^*$ is called the {\em context-free language} (CFL) generated by $G$.
\end{defn}

We will now state some closure properties of CFLs.

\begin{theorem}[{\cite[Theorem~7.27]{hopcroft}}]\label{T:preimage CFL}
Let $\C L \subseteq Y^*$ be a CFL, and let  $\phi \colon X^* \to Y^*$ be a monoid homomorphism. Then $\phi^{-1}(\C L)$ is also a CFL. \qed
\end{theorem}

\begin{theorem}[{\cite[Theorem~7.30]{hopcroft}}]\label{T:regular cap CFL}
If $\C R \subseteq X^*$ is regular and $\C L \subseteq X^*$ is a CFL then $\C R \cap \C L$ is also a CFL. \qed
\end{theorem}

\begin{defn}\label{D:D(S)}
Let $\Phi \colon S^* \rightarrow G$ be the canonical projection with $\Phi(s) = s$ for all $s \in S$. The \textit{cancellation language}, $\C W(G,S)$, is defined to be
$$\C W(G,S) =\Phi^{-1}(1) \subseteq S^*.$$
\end{defn}

We now state the following classical theorem.

\begin{theorem}[{Muller-Schupp, \cite[Theorem~III]{mschupp}, \cite[Section~6.1]{trees}}]\label{T:mschupp} The language $\C W(G, S)$ is a CFL if and only if $G$ is a virtually free group. \qed
\end{theorem}

\begin{defn}
Let $X$ be a finite set. The \textit{Parikh homomophism} is the abelianisation map, denoted
$$\psi_X \colon X^* \to \B N^X.$$
It follows directly from the definition that for any $w \in X^*$, $\psi(w)(x)$ is equal to the number of occurrences of the letter $x$ in $w$. Therefore, for any $w \in X^*$,
\begin{equation}\label{E:Psi=sum circ Psi}
|w| = \sum_{x \in X} \psi_X(w)(x).
\end{equation}
\end{defn}

\begin{remark}\label{usefulrem22}
Considering that $\B N^{X \sqcup Y} \cong \B N^X \times \B N^Y$, the Parikh homomorphism $\psi_{X \sqcup Y} \colon (X \sqcup Y)^* \to \B N^{X \sqcup Y}$ has the form
$$
\psi_{X \sqcup Y} = (\psi_X \circ \pi_X, \psi_Y \circ \pi_Y).
$$
\end{remark}

An important result in the theory of CFLs is Parikh's theorem.

\begin{theorem}[{Parikh, \cite{parikh}}]\label{T:Parikh}
Let $\C L \subseteq X^*$ be CFL.
Then $\psi_X(\C L)$ is semilinear. \qed
\end{theorem}

\begin{defn}\label{D:Henry's R}
Let $\C L_1, \hdots,\C L_n \subseteq X^*$ be languages, and further let $Y = \{y_1,\hdots,y_n\}$ be a set disjoint from $X$. 
Define the \textit{enumeration language} of $\C L_1, \hdots,\C L_n$ in $(X \sqcup Y)^*$ to be
$$\C E = (\C L_1 \cdot \{y_1\})^* \cdot \hdots \cdot (\C L_n \cdot \{y_n\})^*.$$
\end{defn}

For simplicity, let us identify $\mathbf{N}^Y$ with $\mathbf{N}^n$.
Observe that $\C E$ is the disjoint union
\begin{equation}\label{E:partition R}
\C E = \bigsqcup_{\mathbf{k} \in \mathbf{N}^Y} \C E_{\mathbf{k}},
\end{equation}
where for every $\mathbf{k} = (k_1, \hdots, k_n) \in \B N^Y$,
$$
\C E_{\mathbf{k}} = (\C L_1 \cdot \{y_1\})^{k_1} \cdot \hdots \cdot (\C L_n \cdot \{y_n\})^{k_n}
$$
is the preimage of $\mathbf{k}$ with respect to $\psi_Y \circ \pi_Y \colon \C E \to \B N^Y$.
Moreover,
\begin{equation}\label{E:phi_X surjective on R_k}
\pi_X(\C E_{\mathbf{k}}) = \C L_1^{k_1} \cdot \hdots \cdot \C L_n^{k_n}.
\end{equation}

\begin{remark}
Enumeration languages can be considered as graded unions of combinatorial cubes, in the sense of \cite[Section~3]{combi}.
\end{remark}

\section{Multivariate stability}\label{S:multivariate stability}

The purpose of this section is to prove the main result of the paper.

\begin{theorem}\label{T:multivariate CFL relative length}
Let $\C L \subseteq X^*$ be a CFL, and let $w_1,\hdots,w_n \in X^*$. Then the sequence 
$
(\|w_1^k \hdots w_n^k\|_{\C L})_k
$
is uniformly semi-arithmetic in the sense of \Cref{D:semi-arithmetic}. In particular, if $(\|w_1^k \hdots w_n^k\|_{\C L})_k$ is eventually finite, and has period $m > 0$ and difference $d \geq 0$, then
$$
\lim_{k \to \infty} \frac{\|w_1^k \hdots w_n^k\|_{\C L}}{k}
$$
exists and is equal to $\frac{d}{m} \in \B Q$.

\end{theorem}

\begin{proof}
Let $U \subseteq \B N^2$ be defined by
$$
U \col= \{(k, |w_1^k \hdots w_n^k|-|u|): u \in \mathcal{P}(w_1^k \hdots w_n^k) \cap \mathcal{L}, \, k \in \mathbf{N}\}.
$$
Observe that by \Cref{D:L pseudo-norm},
$$U_{(k)} = \{|w_1^k \hdots w_n^k| - |u|: u \in \mathcal{P}(w_1^k \hdots w_n^k) \cap \mathcal{L}\},$$
and hence by \Cref{D:Env},
\begin{equation}\label{important}
    \OP{Env}(U) = (\|w_1^k \hdots w_n^k\|_{\C L})_k.
\end{equation}

Let $Y=\{y_1,\hdots,y_n\}$ be disjoint from $X$, and consider the enumeration language
\begin{equation} \label{henry R}
    \C R = (\mathcal{P}(w_1)\cdot \{y_1\})^* \cdot \hdots \cdot (\mathcal{P}(w_n)\cdot \{y_n\})^* \subseteq (X \sqcup Y)^* 
\end{equation}
in the sense of \Cref{D:Henry's R}. By \Cref{D:Regular language}, $\C R$ is a regular language.
With $\pi_X \colon (X \sqcup Y)^* \to X^*$ 
in the sense of \Cref{proj}, set
\begin{equation}\label{keyeq}
\C M \col= \C R \cap \pi_X^{-1}(\C L).
\end{equation}
\Cref{T:preimage CFL,T:regular cap CFL} imply that $\C M$ is a CFL in $(X \sqcup Y)^*$.
By \Cref{T:Parikh}, 
$$
M \col= \psi_{X \sqcup Y}(\C M) \subseteq \B N^{X \sqcup Y}
$$
is a semilinear set.

Consider the diagonal $\Delta_Y  \subseteq \B N^Y$ as defined in \Cref{E:diagonal}. Then by \Cref{E:diagonal} and \Cref{products}, $\B N^X \times \Delta_Y$ is a finitely-generated linear subset of $\B N^{X \sqcup Y}$.
By \Cref{T:intersection semilinear},
$$
M_\Delta \col= M \cap (\B N^X \times \Delta_Y)
$$
is semilinear in $\B N^{X \sqcup Y}$. Now define a monoid homomorphism 
\begin{align*}
    \xi\colon \B N^{X \sqcup Y} &\to \B Z^2\\
         (f, \mathbf{k}) &\mapsto \big(k_1, \sum_{i = 1}^n k_i \cdot |w_i| - \sum_{x \in X} f(x)\big)
\end{align*}
where we identify $\B N^{X\sqcup Y}$ with $\B N^X \times \B N^Y$. \newline

\begin{claim}
   $\xi(M_\Delta) = U$.
\end{claim}\newline

\begin{claimproof}
    Observe that $(k, \ell) \in \xi(M_\Delta)$ if and only if 
    $$
    \exists \, u' \in \C M : \psi_{X \sqcup Y}(u') \in \B N^X \times \Delta_Y \text{ and } \xi(\psi_{X \sqcup Y}(u'))=(k,\ell). 
    $$
    Resolving coordinates, and by both \Cref{usefulrem22} and the definition of $\xi$, this is equivalent to
    $$
        \exists \, u' \in \C M : \sum_{i =1}^n k \cdot |w_i| - \sum_{x \in X} \psi_X(\pi_X(u')) = \ell \text{ and } \psi_Y(\pi_Y(u')) =k \cdot \mathbf{1}_Y.
    $$
    By \eqref{lengths}, \eqref{E:Psi=sum circ Psi} and \eqref{E:partition R} this is equivalent to
    $$
    \exists \, u' \in \C M : |w_1^k \hdots w_n^k|-|\pi_X(u')|= \ell \text{ and } u' \in \C R_{k \cdot \mathbf{1}_Y}.
$$
Since $\C M=\C R \cap \pi_X^{-1}(\C L)$ and by \eqref{E:subwords concatenation} and \eqref{E:phi_X surjective on R_k}, this is equivalent to
$$
\exists \, u \in \mathcal{P}(w_1^k\hdots w_n^k) \cap \C L : |w_1^k\hdots w_n^k|-|u|=\ell
$$
which is equivalent to $(k,\ell) \in U$, completing the proof of the claim.
\end{claimproof}
\newline
    
Finally, since $M_\Delta$ is semilinear in $\B N^{X \sqcup Y}$ and $\xi$ is a homomorphism of monoids, it follows that $\xi(M_\Delta)$ is the union of finitely many linear subsets of $\B Z^2$.
Since $\xi(M_\Delta)=U$ is contained in $\B N^2$, this implies that $U \subseteq \B N^2$ is semilinear. Hence, by \ref{important} and \Cref{C:Env of semilinear}, $(\|w_1^k \hdots w_n^k\|_{\C L})_k$ is semi-arithmetic.

For any $k \geq 0$ the word $w_1^k \hdots w_n^k$ can be obtained from $w_1^{k+1} \hdots w_n^{k+1}$ by deleting $D\col=\sum_i|w_i|$ symbols.
By the definition of $\| \cdot \|_{\C L}$,
$$
\|w_1^{k+1} \hdots w_n^{k+1}\|_{\C L} \leq \|w_1^{k} \hdots w_n^{k}\|_{\C L} + D.
$$
Hence, by \Cref{L:usa criterion}, $(\|w_1^k \hdots w_n^k\|_{\C L})_k$ is a uniformly semi-arithmetic sequence.
The final statement of the theorem follows from \Cref{L:usa limit}.
\end{proof}

\begin{proof}[Proof of \Cref{T:main}]
Let $\Phi \colon S^* \to G$ be the canonical projection as defined in \Cref{D:D(S)}. Now consider some $g_1,\hdots, g_n \in G$ and choose  $w_1,\hdots,w_n \in S^*$ such that $\Phi(w_i)=g_i$. Since the conjugation-invariant length is equal to the cancellation length, 
$$
\|g_1^k \hdots g_n^k\|_{S} = \|w_1^k \hdots w_n^k\|_{\C W(G,S)},
$$
and in particular, by \Cref{T:multivariate CFL relative length}, $(\|g_1^k \hdots g_n^k\|_{S})_{k}$ is uniformly semi-arithmetic sequence $(a_k)_k$ with period $m > 0$ and difference $d \geq 0$.
Since $\|g\|_{S}<\infty$ for all $g \in G$, the sequence $(\|g_1^k \hdots g_n^k\|_{S})_k$ is eventually finite, and hence 
$$\lim_{k \to \infty} \frac{\|g_1^k \hdots g_n^k\|_{S}}{k}$$
exists and converges to $\frac{d}{m}$.
\end{proof}

\bibliography{ref}
\bibliographystyle{plain}

\end{document}